\documentclass[12pt,a4paper]{amsart}
\usepackage{mathrsfs}
\usepackage{amsfonts}
\usepackage{txfonts}
 \usepackage{color}
\usepackage{hyperref}
\usepackage{latexsym}
\usepackage{amssymb}

\newtheorem{theorem}{Theorem}[section]
\newtheorem{lemma}[theorem]{Lemma}
\newtheorem{corollary}[theorem]{Corollary}
\newtheorem{proposition}[theorem]{Proposition}
\newtheorem{example}[theorem]{Example}

\theoremstyle{definition}
\newtheorem{definition}[theorem]{Definition}

\newtheorem{remark}[theorem]{Remark}

\newcommand{\R}{\mathbb{R}}
\newcommand{\N}{\mathbb{N}}
\newcommand{\eps}{\varepsilon}

\numberwithin{equation}{section}

\begin{document}

\title[On convex bodies]
{{\bf On Klee's problem of convex bodies in Banach spaces}}

\author{ Lixin Cheng, Chunlan Jiang, Liping Yuan }
\address{  Lixin Cheng$^\dag$:  School of Mathematical Sciences, Xiamen University,
 Xiamen 361005, China}
\address{ Chunlan Jiang$^\ddag$:  School of Mathematical Sciences, Hebei Normal University,
 Shijiazhuang 050024, China}
 \address{Liping Yuan$^\natural$: School of Mathematical Sciences, Hebei Normal University, Shijiazhuang 050024, China}
 \email{$^\dag$: lxcheng@xmu.edu.cn\;\;(Lixin Cheng)}
 \email{$^\ddag$: cljiang@hebtu.edu.cn\;\;(Chunlan Jiang)}
 \email{$^\natural$: lpyuan@hebtu.edu.cn \;\;(Liping Yuan) }

\thanks{$^\dag$ Support in partial
by the Natural Science Foundation of China, grant no. 11731010.\\\indent $^\ddag$ Support in partial
by the Natural Science Foundation of China, grant no.  11471271 \& 11471270.
\\\indent$^\natural$ Support in partial
by the Natural Science Foundation of China, grant no. 12271139 and the Hebei Natural Science
Foundation, grant no.  A2023205045.                      }

\date{}

\begin{abstract}
  It is well known that every convex body in a finite dimensional normed space can be uniformly approximated by strictly convex and smooth convex bodies.  However, in the case of infinite dimensions, little progress has been made since Klee asked  how it is in the case of infinite dimensions in 1959. In this paper,  we show that for an infinite dimensional Banach space $X$, (1) every convex body can be uniformly approximated by strictly convex  bodies if and only if $X$ admits an equivalent strictly convex norm; (2) every convex body can be uniformly approximated by G\^{a}teaux smooth convex  bodies if the dual $X^*$ of $X$ admits an equivalent strictly convex dual norm;  in particular, (3) if $X$ is either separable, or reflexive, then every convex body in $X$ can be uniformly approximated by strictly convex and smooth convex bodies.  They are done by showing that some correspondences among the sets of all convex bodies endowed with the Hausdorff metric, all continuous coercive Minkowski functionals and Fenchel's transform defined on all quadratic homogenous continuous
convex functions  equipped with the metric induced by the sup-norm of all bounded continuous functions defined on the closed unit ball $B_X$ are actually locally Lipschitz isomorphisms.
\end{abstract}

\keywords{ convex body, strict convexity, smoothness, Banach space}

\subjclass{}

\maketitle

\section{Introduction}
Let $X$ be a Banach space and $X^*$ its dual. $B_X$ stands for the closed unit ball of $X$. We use $\mathfrak C(X)$, $\mathfrak K(X)$, $\mathfrak C_{\rm sc}(X)$ and $\mathfrak C_{\rm sm}(X)$,  to denote  successively, the  cone of all nonempty bounded closed convex sets,  all nonempty compact convex sets,  all strictly convex bodies and all G\^{a}teaux smooth convex bodies endowed with the Hausdorff metric (also called Pompeiu-Hausdorff metric) $d_{H}$ defined for every pair of nonempty bounded convex sets $A, B\subset X$ by \[d_{H}(A, B)=\inf\{r>0: A\subset B+rB_X,\;B\subset A+rB_X\}.\]
Since every $d$ dimensional real normed space $X$ is isometric to $(\R^d,\|\cdot\|)$ for some equivalent norm $\|\cdot\|$ on $\R^d$, in the following we often blur the distinction between a $d$ dimensional normed space $X$ and $\R^d$.

In 1959, V. Klee  \cite{Klee1959}  showed that  in $\R^d$, the set of convex bodies which are both strictly convex and smooth, i.e. $\mathfrak C_{\rm sc}(\R^d)\bigcap\mathfrak C_{\rm sm}(\R^d)$ contains a dense $G_\delta$ subset of $\mathfrak C(\R^d)$. As a consequence of this result, it reveals the following important fact:  \emph{Every convex body in $\R^d$ can be uniformly approximated by strictly convex and smooth convex bodies.}
Meanwhile, he further pointed out ``{\it we have not succeeded in answering some of the natural questions about infinite-dimensional convex bodies}"\cite{Klee1959}.

 We would like to mention that,  in the case of finite dimensions, P. Gruber \cite{gruber1977} (1977), among many other things, rediscovered Klee's theorem. In 1987, T. Zamfirescu gave another remarkable result:  \emph{The set of strictly convex and smooth convex bodies, i.e.  $\mathfrak C_{sc}(\R^d)\bigcap\mathfrak C_{\rm sm}(\R^d)$ cantains the complement of a $\sigma$-porus set of $\mathfrak C(\R^d)$ (hence, it contains a dense a dense $G_\delta$ subset).} Also, a number of interesting geometrical results have been proven  in $\mathbb{R}^d$, so for example,  results on most convex bodies (called generic), regarded so far differentiability of the boundary \cite{zm1980-2, barany2015, Schneider2015}, geodesics \cite{zm1982, zm1992}, diameters \cite{bz1990, bl2017}, etc.
However, in the case of infinite dimensions, little progress has been made since 1959. 
In response to Klee's problem,  in this paper, we show the following theorems.

\begin{theorem}\label{1.1}  Every convex body in a Banach space $X$ can be uniformly approximated by strictly convex bodies
if and only if there exists an equivalent strictly convex norm on $X$.
\end{theorem}
\begin{theorem}\label{1.2}
Let $X$ be a Banach space.

i) If there is an equivalent norm on $X$  so that its dual $X^*$ is strictly convex, then every convex body in $X$ can be uniformly approximated by G\^{a}teaux smooth convex bodies.

ii) If $X$ satisfies that every convex body in $X$ can be uniformly approximated by G\^{a}teaux smooth convex bodies, then $X$ admits an equivalent G\^{a}teaux smooth norm.
\end{theorem}

But we do not know whether the converse of ii) is true.

\begin{theorem}\label{1.3}
Let $X$ be a Banach space.
If there is an equivalent strictly convex norm on $X$  so that its dual $X^*$ is strictly convex, then every convex body in $X$ can be uniformly approximated by strictly convex and G\^{a}teaux smooth convex bodies.
\end{theorem}

In particular,

\begin{theorem}\label{1.4}
If $X$ is a separable, or, reflexive Banach space, then every convex body in $X$ can be uniformly approximated by strictly convex and G\^{a}teaux smooth convex bodies.
\end{theorem}



In order to outline the ideas and methods of our proof, we need more symbols. We use  $\mathfrak C_{00}(X)$ to denote all convex bodies $B$ of $X$ with the origin $0\in{\rm int}(B)$. $\mathscr M_0(X)$ ($\mathscr H_c^2(X)$, respectively) stands for the cone of all continuous coercive Minkowski functionals (all positive quadratic homogenous continuous coercive convex functions, respectively)  defined on $X$ endowed with the metric $d$ defined for $f,g\in\mathscr M_0(X)$ ($\mathscr H_c^2(X)$, respectively) by
 $d(f,g)=\sup_{x\in B_X}|f(x)-g(x)|$, 
 where ``$f$ is coercive" means $f(x)\rightarrow+\infty$ whenever $\|x\|\rightarrow+\infty$. Finally, $\mathscr H_c^{*2}(X)$ represents the cone of all continuous and $w^*$-lower semicontinuous positive quadratic homogenous coercive convex functions on $X^*$.

 To show the theorems mentioned above, we establish a number of correspondences among the cones  mentioned previously, and then show the following correspondences are order-reversing and locally Lipschitz isomorphisms.

 i) the correspondence $\mathfrak C_{00}(X)\rightarrow\mathscr M_0(X)$ by $C\rightarrow p_C$, the Minkowski functional generated by $C\in \mathfrak C_{00}(X)$;

 ii) the correspondence $\mathscr H_c^2(X)\rightarrow \mathfrak C_{00}(X)$;

 iii) Fenchel's transform from $\mathscr H_c^2(X)$ to  $\mathscr H_c^{*2}(X)$.

\section{Preliminaries}

To begin with, we recall some definitions and basic properties which will be used in the sequel.

\subsection{Convex body}

Let $X$ be a real Banach space.   A hyperplane $H$ of $X$ is an affine subspace of codimension one, i.e. there exist a maximal proper subspace $N$ of $X$ and $x_0\in X$ so that $H=x_0+N$. Note that for each closed hyperplane $H$ there is a continuous functional $x^*(\neq0)\in X^*$ so that $H=\{x\in X: \langle x^*,x\rangle=1\}$ if $H$ is  off the origin; and  $H=\{x\in X: \langle x^*,x\rangle=0\}$ if $H$ is a linear subspace. A hyperplane $H$ is said to support a set $B$ at a point $x\in {\rm bdr}B$ if $x \in H$ and if $H$ bounds $B$.

A closed bounded convex set $B\subset X$ is said to be a {\it convex body} provided its interior ${\rm int}B\neq\emptyset$.
A convex body $B$ is called {\it strictly convex} if every boundary point $x$ of $B$ is an extreme point of $B$, that is, $x$
is not the middle point of any line-segment in $B$. We say  a convex body $B$ is {\it (G\^{a}teaux) smooth} if it satisfies that every $x\in{\rm bdr}(B)$ (the boundary of $B$) admits a unique supporting hyperplane (closed hyperplane) of $B$ through $x$.

We use $B_X$ and $S_X$ to denote the closed unit ball and the unit sphere of $X$, i.e. $B_X=\{x\in X: \|x\|\leq1\}$, and $S_X=\{x\in X: \|x\|=1\}$. We say that $X$ is  {\it strictly convex} (resp. {\it G\^{a}teaux smooth}) if the closed unit ball $B_X$ (as a convex body) is strictly convex (G\^{a}teaux smooth, resp.).

\subsection{Minkowski functional}

Let $C$ be a  closed convex set of a Banach space X with $0\in{\rm int}C$. The functional $p_C$ defined for $x\in X$ by $p_C(x)=\inf\{\lambda>0: x\in\lambda C\}$ is called the {\it Minkowski functional} with respect to (or, generated by) $C$.


\begin{proposition}\label{2.3}
Every Minkowski functional $p_C$ defined on a Banach space $X$ is a nonnegative-valued continuous sublinear functional, that is,
$$p_C(kx)=kp_C(x),\;\;p_C(x+y)\leq p_C(x)+p_C(y),\;\;\forall\;x,y\in\;X\;{\rm and}\;k\geq0.$$

\end{proposition}

For  a convex set $C$ of a Banach space $X$,  let $L_C=\{x\in X:x+C=C\}$, $K_C=\{x\in X: x+C\subset C\}$. Then $C$ is said to be {\it line-free} ({\it ray-free,} resp.) provided $L_C=\{0\}$ (
$K_C=\{0\}$, resp.).

\begin{proposition}\label{2.5}
Let $C$ be a  closed convex set of X with $0\in{\rm int}C$. Then

i) $p_C$ is strictly positive, i.e. $p_C(x)=0 \Longrightarrow x=0$, if and only if $C$ is ray-free.

ii) $p_C$ is a seminorm on $X$ if and only if $C$ is symmetric, i.e. $-C=C$;

iii) $p_C$ is a norm if and only if $C$ is symmetric and line-free.

iv) $p_C$ is an equivalent norm on $X$  if and only if $C$ is a symmetric convex body.
\end{proposition}

\subsection{Differentiability}

Let $f$ be a real-valued function defined on a Banach space $X$, and $x\in X$ be fixed. Then $f$ is said to be {\it G\^{a}teaux differentiable at $x$} provided there is a continuous functional $x^*\in X^*$ such that

\begin{equation}\label{2.3.1}\lim_{t\rightarrow 0}\frac{f(x+ty)-f(x)}t=\langle x^*,y\rangle,\;\;\forall\;y\in X.\end{equation}
In this case, we denote by $d_Gf(x)=x^*$, and call it the {\it G\^{a}teaux derivative of $f$ at $x$}.

  $f$ is said to be {\it Fr\'{e}chet differentiable at $x$} if there is a continuous functional $x^*\in X^*$ such that

\begin{equation}\label{2.3.2}\lim_{\|y\|\rightarrow 0}\Big(\frac{f(x+y)-f(x)-\langle x^*,y\rangle}{\|y\|}\Big)=0.\end{equation}
In this case, we denote by $d_Ff(x)=x^*$, and call it the {\it Fr\'{e}chet derivative of $f$ at $x$}.

The norm $\|\cdot\|$ of a Banach space $X$ is said to be {\it G\^{a}teaux (Fr\'echet, {\rm resp.}) smooth} provided it is everywhere G\^{a}teaux ( Fr\'echet, resp.) differentiable off the origin.

\begin{theorem}[Mazur's theorem](See, for instance, \cite{Ph}.)
Every continuous convex function defined on a separable Banach space is everywhere G\^{a}teaux differentiable in a dense $G_\delta$-subset of the space.
\end{theorem}
\begin{theorem}[Asplund's theorem](See,\cite{Asp1}, also \cite{Ph}.)
Every continuous convex function defined on a separable Banach space with separable dual is  Fr\'echet differentiable everywhere in a dense $G_\delta$-subset of the space.
\end{theorem}

\subsection{Subdifferentiability}
Let $f$ be a convex function defined on a Banach space $X$. Its {\it subdifferential mapping} $\partial f: X\rightarrow 2^{X^*}$ is defined for $x\in X$ by
\begin{equation}\label{2.4.1}\partial f(x)=\{x^*\in X^*: f(y)-f(x)\geq\langle x^*,y-x\rangle,\;y\in X\}.\end{equation}

Let $(U,\tau_U), (V,\tau_V)$ be two topological spaces. A set-valued mapping $P: U\rightarrow 2^V$ is said to be {\it upper-semicontinuous} provided that for each $u\in U$ and each open neighborhood $O_{P(u)}$ of $P(u)$, there is an open neighborhood $O_u$ of $u$ such that $P(O_u)\subset O_{P(u)}$.   We say that the subdifferential mapping $\partial f: X\rightarrow 2^{X^*}$ defined above is {\it norm-to-$w^*$ upper semicontinuous} if we endow $X$ with the norm-topology and its dual space $X^*$  with its weak-star   topology.

The following properties regarding to differentiability and subdifferentiability of convex functions are classical. See, for instance, \cite{Ph}.
\begin{proposition}\label{2.4.2}
Let $f$ be a continuous convex function defined on a Banach space $X$. Then

i) for each $x\in X$, $\partial f(x)$ is a nonempty $w^*$-compact convex set of $X^*$;

ii) $\partial f: X\rightarrow 2^{X^*}$  is norm-to-$w^*$ upper semicontinuous;

iii) $f$ is G\^{a}teaux differentiable at $x$ if and only if $\partial f(x)$ is a singleton;

iv) $f$ is Fr\'echet differentiable at $x$ if and only if $\partial f(x)$ is  single-valued and norm-to-norm upper semicontinuous at $x$.
\end{proposition}

Recall that a selection  for $\partial f$ is a single-valued mapping $\varphi: X\rightarrow X^*$ satisfying $\varphi(x)\in\partial f(x)$ for all $x\in X$.

\begin{proposition}\label{2.4.3}
Suppose that $f$ is a continuous convex function defined on a Banach space $X$. Then the following statements are equivalent.

i) $f$ is G\^{a}teaux differentiable at $x$;

ii) $\partial f$ is single-valued at $x$;

iii) every selection for $\partial f$ is norm-to-$w^*$ continuous at $x$;

iv) there is a selection for $\partial f$ which is norm-to-$w^*$ continuous at $x$.
\end{proposition}
\begin{proposition}\label{2.4.4}
Suppose that $f$ is a continuous convex function defined on a Banach space $X$. Then the following statements are equivalent.

i) $f$ is Fr\'echet differentiable at $x$;

ii) $\partial f$ is single-valued at $x$ and norm-to-norm upper semicontinuous at $x$.;

iii) every selection for $\partial f$ is norm-to-norm continuous at $x$;

iv) there is a selection for $\partial f$ which is norm-to-norm continuous at $x$.
\end{proposition}
\begin{proposition}\label{2.4.5}
Let $C$ be a convex set of $X$ containing $0$ in its interior.  Then $x^*\in \partial p_C(x)$ if and only if the following two conditions are satisfied:

i) $\langle x^*,z\rangle\leq p_C(z),\;\;\forall\;z\in X;$

ii) $\langle x^*,x\rangle= p_C(x)$.
\end{proposition}

\subsection{Fenchel's transform}
We say an extended real-valued convex function $f$ is {\it proper}, if $f(x)>-\infty$ for all $x\in X$ and its essential domain ${\rm dom}f\equiv\{x\in X: f(x)<+\infty\}\neq\emptyset$.  For a Banach space $X$, we denote by $\mathscr C_{\rm conv}(X)$ the set of all extended real-valued lower semicontinuous proper convex functions defined on $X$.

\begin{definition}\label{2.5.1}
Fenchel's transform $\mathscr F$ is defined for $f\in \mathscr C_{\rm conv}(X)$ by
\begin{equation}\label{2.5.2}
\mathscr F(f)(x^*)=\sup\{\langle x^*,x\rangle-f(x): x\in X\}, \;\;x^*\in X^*.
\end{equation}
\end{definition}

If $X$ is of finite dimension, then  $X^*$ is (linearly) isomorphic to $X$. Therefore, in this case,  Fenchel's transform $\mathscr F$ defined above is actually from  $\mathscr C_{\rm conv}(X)$ to itself, we denote it by $\mathscr L$ and call it Legendre's transform.

\begin{proposition}\label{2.5.3} Let $X$ be a Banach space. Then
 for each  $f\in \mathscr  C_{\rm conv}(X)$, $\mathscr F(f)$ is a $w^*$-lower semicontinuous (proper) convex function.
 \end{proposition}
 \begin{proof}
Note that each $f\in\mathscr C_{\rm conv}(X)$ is lower semicontinuous and proper. Then  by the Br{\o}dsted-Rockafellar Theorem (see, for instance, \cite{Ph}), there exists $x_0\in{\rm dom}(f)\equiv\{x\in X: f(x)<\infty\}$ such that $\partial f(x_0)\neq\emptyset$. Fix any $x^*_0\in \partial f(x_0)$. Since $f(x)-f(x_0)\geq\langle x^*_0,x-x_0\rangle$ for all $x\in X$, we have
  \[\mathscr F(f)(x^*_0)=\sup\{\langle x^*_0,x\rangle-f(x): x\in X\}=\langle x^*_0,x_0\rangle-f(x_0)<+\infty.\]
  Therefore, $\mathscr F(f)$ is a proper convex function on $X^*$. From
 (\ref{2.5.2}), $w^*$-lower semicontinuity of  $\mathscr F(f)$ follows.
 \end{proof}
We use $\mathscr C^*_{\rm conv}(X^*)$ to denote the set of all extended real-valued and $w^*$-lower semicontinuous  (hence, lower semicontinuous)  proper convex functions defined on $X^*$.

The following result is classical (see, for instance, \cite{Di}).

\begin{proposition}\label{2.5.4}
Let $X=(X,\|\cdot\|)$ be a Banach space, and $f=\frac{1}2\|\cdot\|^2$. Then $\mathscr F(f)=\frac{1}2{\|\cdot\|^*}^2$, where $\|\cdot\|^*$ is the dual norm of $X^*$.
\end{proposition}
The notions of order-preserving (order reversing, resp.)  and fully order order-preserving (fully order reversing, resp.) were used in \cite{AAM} \cite{CL} and \cite{CL2}.
Let $A$ and  $B$ be two partially ordered sets. A mapping $M: A\rightarrow B$ is said to be {\it order preserving} ({\it order-reversing}, resp.), provided $a\geq b\in A\Longrightarrow M(a)\geq M(b)$ ($a\geq b\in A\Longrightarrow M(a)\leq M(b)$, resp.). If, in addition, $M$ is a bijection from $A$ to $B$ and is
{order preserving} ({ order-reversing}, resp.), then $M$ is called a {\it fully order-presevering} ({\it fully order-reversing}, resp.,)  mapping.

\begin{proposition}\label{2.5.5} Let $X$ be a Banach space. Then

i)  the Fenchel transform
 $\mathscr F: \mathscr C_{\rm conv}(X)\rightarrow \mathscr C_{\rm conv}(X^*)$ is  order-reversing;

ii) $\mathscr F: \mathscr C_{\rm conv}(X)\rightarrow \mathscr C^*_{\rm conv}(X^*)$ is fully order-reversing;

iii) $\mathscr F: \mathscr C_{\rm conv}(X)\rightarrow \mathscr C_{\rm conv}(X^*)$ is fully order-reversing if and only if $X$ is reflexive.
 \end{proposition}
 \begin{proof}
 i) It follows from (\ref{2.5.2}) and  Proposition \ref{2.5.3}.

 ii) By i),  $\mathscr F: \mathscr C_{\rm conv}(X)\rightarrow \mathscr C^*_{\rm conv}(X^*)$ is order-reversing.   Next, we show that $\mathscr F$ is surjective. Fix $g\in\mathscr C^*_{\rm conv}(X^*)$.
Since $g$ is $w^*$-lower senicontinuous and proper convex on $X^*$, there is $f\in \mathscr C_{\rm conv}(X)$ such that $\mathscr F(f)=g$. It remains to show that
$\mathscr F^{-1}: \mathscr C^*_{\rm conv}(X^*)\rightarrow \mathscr C_{\rm conv}(X)$ is again order-reversing. Let $\mathscr F|_{*}$ be defined for $g\in\mathscr C^*_{\rm conv}(X^*)$ by
\[\mathscr F|_{*}(g)=\mathscr F(g)|_{X},\]
i.e. $\mathscr F(g)$ but its domain is restricted to $X$. Then it is also order-reversing. Note $\mathscr F|_{*}\circ\mathscr F$ is just the identity $I=\mathscr F^{-1}\circ\mathscr F$ on $\mathscr C_{\rm conv}(X)$. Therefore, $\mathscr F^{-1}=\mathscr F|_{*}$ is order-reversing. Consequently, we finish the proof of ii).

iii) Since $\mathscr C^*_{\rm conv}(X^*)=\mathscr C_{\rm conv}(X^*)$ if and only if $X$ is reflexive, iii) follows from ii).
 \end{proof}

\subsection{Inf-convolution}
Let $f,g$ be two continuous positive quadratic homogenous convex functions defined on $X$. Their inf-convolution $f\Box g$ is defined by
\begin{equation}\label{2.6.1} f\Box g(x)=\inf\{f(x-y)+g(y): y\in X\},\;\;x\in X.\end{equation}
\begin{proposition}\label{2.6.2}
Let $f,g\geq0$ be two continuous  positive quadratic homogenous convex functions  defined on $X$, and $\mathscr F$ be Fenchel's transform. Then

i) $\mathscr F(f+g)= \mathscr F(f)\Box\mathscr F(g)$;

ii)$\mathscr F(f\Box g)= \mathscr F(f)+\mathscr F(g)$.
\end{proposition}

\begin{proof} i)\; For every $x^*\in X^*$,
\[\mathscr F(f)\Box\mathscr F(g)(x^*)=\inf_{y^*\in X^*}\{\mathscr F(f)(x^*-y^*)+\mathscr F(g)(y^*)\}\]
\[=\inf_{y^*\in X^*}\{\sup\{\langle x^*-y^*,x\rangle-f(x):x\in X\}+\sup\{\langle y^*,y\rangle-g(y):y\in X\}\}\]
\[\geq\inf_{y^*\in X^*}\sup_{x\in X}\{\langle x^*,x\rangle-(f(x)+g(x)):x\in X\}=\mathscr F(f+g)(x^*). \]

On the other hand,  $x^*\equiv u^*+v^*\in\partial (f+g)(x)=\partial (f)(x)+\partial (g)(x)$ (where $u^*\in\partial f(x), v^*\in\partial g(x)$) implies
\[\mathscr F(f+g)(x^*)=\sup_{y\in X}\{\langle x^*,y\rangle-(f(y)+g(y))\}=\langle x^*,x\rangle-(f+g)(x)\]
\[=(\langle u^*,x\rangle-f(x))+(\langle v^*,x\rangle-g(x))\]
\[=\mathscr F(f)(u^*)+\mathscr F(g)(x^*-u^*)\]\[\geq\mathscr F(f)\Box \mathscr F(f)(x^*).\]
Since $f,g$ are  continuous positive quadratic homogenous convex, by the Bishop-Phelps theorem (see, for instance, \cite{Ph}), the range of $\partial (f+g)$ is dense in $X^*$.
This and continuity of $\mathscr F(f+g)$ and $\mathscr F(f)\Box \mathscr F(f)$ entail that for all $x^*\in X^*,$
\[\mathscr F(f+g)(x^*)\geq\mathscr F(f)\Box \mathscr F(f)(x^*).\]
Therefore, i) is shown.

We can show ii) analogously.

\end{proof}

\section{Strictly convex bodies and strictly convex functions}

A convex function $f$ defined on a convex set $C$ of a Banach space $X$ is said to be {\it strictly convex} provided
\begin{equation}\label{3.1}
\lambda f(x)+(1-\lambda)f(y)>f(\lambda x+(1-\lambda)y),\;\;\forall\;x\neq y\in C, \;0<\lambda<1.
\end{equation}
If, in addition, $f$ is continuous, then (\ref{3.1}) is equivalent to
$$\frac{1}2\Big(f(x)+f(y))>f(\frac{x+y}2),\;\forall x\neq y\in X.$$

A  function $f$ on a Banach space $X$ is said to be {\it coercive}, if it satisfies $\lim_{\|x\|\rightarrow\infty}f(x)=+\infty$. For example, for every nontrivial Banach space $X=(X,\|\cdot\|)$, the norm $\|\cdot\|$ is a coercive convex function.
\begin{lemma}\label{3.2}
Suppose that $f$ is a nonnegative-valued continuous  coercive convex function. Then

i) for each $r>\inf_{x\in X}f(x)$, the level set $B_r\equiv\{x\in X: f(x)\leq r\}$ is a convex body of $X$;

ii) if $f$ is strictly convex then for each such $r$, the convex body $B_r$ is strictly convex.
\end{lemma}
\begin{proof}
i) Clearly, convexity and continuity of $f$ implies that for each  $r>\inf_{x\in X}f(x)$, the level set $B_r\equiv\{x\in X: f(x)\leq r\}$ is a closed convex set with nonempty interior.  Since $f$ is coercive, $B_r$ is bounded.
Therefore, $B_r$ is a convex body.

ii) Suppose that $B_r$ is not strictly convex for some $r>\inf_{x\in X}f(x)$.  Note that the boundary ${\rm bdr}B_r$ of the convex body $B_r$ is $\{x\in X: f(x)=r\}$. Then there exist $x\neq y\in{\rm bdr}B_r$ and $0<\lambda <1$ so that $z\equiv\lambda x+(1-\lambda)y\in{\rm bd}B_r$. Therefore,
$$\lambda f(x)+(1-\lambda)f(y)=r=f(z)=f(\lambda x+(1-\lambda)y).$$
This  contradicts to that $f$ is strictly convex.
\end{proof}
\begin{remark}\label{3.3}
The converse version is not true. For example, assume that $X=(X,\|\cdot\|)$ is a strictly convex Banach space. For each $r>\inf_{x\in X}\|x\|=0$,
$B_r$ is a strictly convex body. But $\|\cdot\|$ is not a strictly convex function. Indeed, let $x\in X$ with $\|x\|=1$ and $y=0$. Then $\|\frac{x+y}2\|=\frac{1}2=\frac{1}2\|x\|+\frac{1}2\|y\|$.  However, the following conclusion holds.
\end{remark}
 Recall that a  convex function $f$ is said to be a  positive quadratic homogenous  if it satisfies $f(kx)=k^2f(x),\;\forall x\in X, k\in\R^+$.

\begin{lemma}\label{3.4}
Let $f$ be a positive quadratic homogenous continuous coercive convex function. Then $f$ is strictly convex if and only if for every $r>0$, the level set $B_r\equiv\{x\in X: f(x)\leq r\}$ is a strictly convex body.
\end{lemma}
\begin{proof}
By Lemma \ref{3.2} ii), it suffices to show the sufficiency. Suppose the contrary that $f$ is not strictly convex. Then there exist $x\neq y\in X$  such that
\begin{equation}\label{3.5} f\Big(\frac{x+y}2\Big)=\frac{1}2\Big(f(x)+f(y)\Big).\end{equation}

Note that both $g(t)\equiv f(tx+(1-t)y)$ and $h(t)\equiv tf(x)+(1-t)f(y)$ are convex functions in  $t\in [0,1]$ with $h\geq g$ on $[0,1]$ and with $h(\frac{1}2)=g(\frac{1}2)$. Since $h$ is affine,
we can see that
\begin{equation}\nonumber
{\rm either}\;\;g(t)\geq h(t),\;\forall t\in[0,\frac{1}2];\;\;{\rm or,}\;g(t)\geq h(t),\; \forall t\in[\frac{1}2,1].
\end{equation}
This and $h\geq g$ imply
\begin{equation}\nonumber
{\rm either}\;\;g(t)=h(t),\;\forall t\in[0,\frac{1}2];\;\;{\rm or,}\;g(t)=h(t),\;\forall t\in[\frac{1}2,1].
\end{equation}
This, $h\geq g$ and affine property of $h$ further entail
\begin{equation}\nonumber
g(t)=h(t),\;\forall t\in[0,1].
\end{equation}
That is,
\begin{equation}\label{3.6}
f(\lambda x+(1-\lambda)y)=\lambda f(x)+(1-\lambda)f(y),\;\forall\;\lambda\in[0,1].
\end{equation}

Since $f$ is a  positive quadratic homogenous continuous coercive convex function, $f(u)>0$ for all $u\neq0$. Therefore, $x,y$ are linearly independent and which entails that $x\neq0\neq y$.
If $f(x)=f(y)\equiv r>0$, then
 (\ref{3.6}) implies that the convex body $B_r$ is not strictly convex, and this is a contradiction.

 Assume $f(y)>f(x)>0$, and let $\alpha=\sqrt{\frac{f(x)}{f(y)}}$, and $z=\alpha y$. Then $0<\alpha<1$ and $f(x)=f(z)\equiv r>0$.
Next, let $0<\lambda<1,\;\gamma=\frac{\lambda}{\lambda+(1-\lambda)\alpha}$.  It follows from (\ref{3.4}) again that
$$f(\lambda x+(1-\lambda)z)=f(\lambda x+(1-\lambda)\alpha y)=\Big(\lambda+(1-\lambda)\alpha\Big)^2f(\gamma x+(1-\gamma)y)$$
$$=\Big(\lambda+(1-\lambda)\alpha\Big)^2\Big(\gamma f(x)+(1-\gamma)f(y)\Big)$$
$$=\Big(\lambda+(1-\lambda)\alpha\Big)^2\Big(\gamma f(x)+\frac{(1-\gamma)}{\alpha^2}f(z)\Big)$$
$$=r\Big(\lambda+(1-\lambda)\alpha\Big)^2\Big(\gamma+\frac{(1-\gamma)}{\alpha^2})\Big).\;\;\;\;\;\;\;\;$$
Let $\lambda=\frac{1}2$. It follows from the equalities above that
$$f(\lambda x+(1-\lambda)z)=r\Big(\lambda+(1-\lambda)\alpha\Big)^2\Big(\gamma+\frac{(1-\gamma)}{\alpha^2})\Big)\;\;\;\;\;\;\;\;$$
$$=r\Big(\frac{1}2+\frac{\alpha}2\Big)^2\Big(\frac{1}{1+\alpha}+\frac{1}{\alpha(1+\alpha)}\Big)
=r\Big(\frac{1}2+\frac{\alpha}{2}\Big)^2\frac{1}\alpha>r.$$
This entails that $\alpha=1$ and consequently, $f(x)=r=f(y)$. This is a contradiction.
\end{proof}
The following lemma directly follows from the definition of strictly convex functions.
\begin{lemma}\label{3.7}
Let $f,g$ be two convex functions defined on a Banach space $X$. If one of $f$ and $g$ is strictly convex, the $f+g$ is again strictly convex.
\end{lemma}
\begin{lemma}\label{3.8}
Let $B\subset X$ be a convex body. Then the following statements are equivalent.

i) $B$ is strictly convex.

ii) For each fixed $z\in{\rm int}B$, $f=\frac{1}2 p_z^2$ is strictly convex.

iii) There is $z\in{\rm int}B$ so that $f=\frac{1}2 p_z^2$ is strictly convex, where $p_z$ is the Minkowski functional generated by $B-z$.

\end{lemma}
\begin{proof}
i)$\Longrightarrow$ ii). Since $B$ is strictly convex, every point of ${\rm bdr}(B)$ is an extreme point of $B$. Therefore, for each fixed $z\in X$, every point of ${\rm bd}(B-z)$ is an extreme point of $B-z$. In particular,
given $x_0\in{\rm int}B$, every point of ${\rm bdr}(B-x_0)$ is an extreme point of $B-x_0$. Let $p_0$ be the Minkowski functional generated by $B-x_0$. Then $\{x\in X:p_0(x)=1\}={\rm bdr}(B-x_0)$. Strict convexity of $B$ entails that each level set of $f\equiv\frac{1}2p_0^2$ is a strictly convex body. By Lemma \ref{3.4},  $f\equiv\frac{1}2p_0^2$ is  strictly convex.

ii) $\Longrightarrow$ iii). Trivial.

iii) $\Longrightarrow$ i). Since $f=\frac{1}2 p_z^2$ is strictly convex, by Lemma \ref{3.4}, $B-z$ is a strictly convex body. Therefore, $B-z$ is again a strictly convex body.
\end{proof}
\section{Smooth convex bodies and G\^{a}teaux differentiable convex functions}
\begin{lemma}\label{4.1}
Let $B\subset X$ be a convex body. Then the following statements are equivalent.

i) $B$ is G\^{a}teaux smooth;

ii) For each fixed $z\in{\rm bdr}(B)$, $p_z$ is everywhere  G\^{a}teaux differentiable off the origin, where $p_z$ is the Minkowski functional generated by $B_z$;

iii) For each fixed $z\in{\rm bdr}(B)$, $f=\frac{1}2p^2_z$ is everywhere  G\^{a}teaux differentiable;

iv) There is $z\in{\rm bdr}(B)$ so that $p_z$ is everywhere  G\^{a}teaux differentiable off the origin;

v) There is $z\in{\rm bdr}(B)$ so that  $f=\frac{1}2p^2_z$ is everywhere  G\^{a}teaux differentiable.

\end{lemma}
\begin{proof}
Recall that $B$ is G\^{a}teaux smooth if and only if for each $u\in{\rm bdr}(B)$ there is a unique closed hyperplane $H$ so that $H\bigcap{\rm int}(B)=\emptyset$ and  $u\in H\bigcap B$. Since $B$ is G\^{a}teaux smooth if and only if $B-z$ is G\^{a}teaux smooth for each $z\in X$, without loss of generality, we assume that $0\in{\rm int}B$.

i) $\Longrightarrow$ ii).
 Clearly, ${\rm bdr}(B)=\{x\in X: p_B(x)=1\}.$ Fix $u\in {\rm bdr}(B)$. Let $H_u$ be the unique supporting hyperplane of $B$ at $u$. By the separation theorem of convex sets, there exists $x^*\in X^*$ so that $$\langle x^*,x\rangle<\langle x^*,u\rangle=1=\sup_{z\in B}\langle x^*,z\rangle, \forall x\in{\rm int}(B).$$
This implies that $H=\{z\in X: \langle x^*,z\rangle=1=\langle x^*,u\rangle=p_B(u)\}$. It is easy to check that $x^*\in\partial p_B(u)$. Since every $u^*\in\partial p_B(u)$ satisfies $\langle u^*,u\rangle=p(u)=1$ and $\langle u^*,x\rangle<1$ for all $x\in{\rm int}(B)$, the uniqueness of $H$ entails $H=\{z\in X: \langle u^*,z\rangle=1\}$. Consequently, we obtain $u^*=x^*$, that is, $\partial p_B(u)$ is a singleton. Thus, ii) follows from Proposition \ref{2.4.2}.

ii) $\Longrightarrow$ iii). It suffices to note that for each $z\in{\rm bdr}(B)$, $p_z$ is everywhere  G\^{a}teaux differentiable off the origin entails that $p^2_z$ is everywhere  G\^{a}teaux differentiable in $X$.

iii) $\Longrightarrow$ iv).  It suffices to note that for each $z\in{\rm bdr}(B)$, $p^2_z$ is everywhere  G\^{a}teaux differentiable  entails that $p_z$ is everywhere  G\^{a}teaux differentiable off the origin.

iv) $\Longrightarrow$ v). It suffices to note that  $p_z$ is everywhere  G\^{a}teaux differentiable off the origin implies that $p^2_z$ is everywhere  G\^{a}teaux differentiable.

v) $\Longrightarrow$ i). Since $f=\frac{1}2p^2_z$ is everywhere  G\^{a}teaux differentiable implies that $p_z$ is everywhere  G\^{a}teaux differentiable off the origin, we obtain that for each $u\in{\rm bd}(B-z)$, there is a unique $x^* (=\partial p_z(u))\in X^*$ so that $\langle x^*,u\rangle=1>\langle x^*,x\rangle$ for all $x\in{\rm int}(B)$. Therefore, $H\equiv\{v\in X: \langle x^*,v\rangle=1\}$ is the unique closed hyperplane satisfying $H\bigcap{\rm int}(B)=\emptyset$ and $u\in H$. Thus, $B-z$ is a G\^{a}teaux smooth convex body.

\end{proof}
\section{Strictly  convex bodies in $\mathfrak C(X)$}
In this section, we will show that for a Banach space $X$, strictly convex bodies are dense in $\mathfrak C(X)$ if and only if $X$ admits an equivalent strictly convex norm, i.e. Theorem  \ref{1.1} stated in the first section.


Let $\mathfrak C(X)$ be the cone of all convex bodies of $X$ with the following two ``linear" operations
\[\alpha B=\{\alpha b: b\in B\}, \;\alpha\in \R, B\in\mathfrak C(X),\]
\[A\oplus B=\overline{\{a+b:a\in A,\;b\in B\}},\;A,B\in\mathfrak C(X)\]
and endowed with the Hausdorff metric $d_H$, 
where $\overline{ D}$ denotes the norm closure of $D\subset X$.
The following property is classical and  easy to prove.
\begin{proposition}\label{5.1}
For every Banach space $X$, $\mathfrak C(X)=(\mathfrak C(X),d_H)$ is a complete metric cone.
\end{proposition}
Let $\mathfrak C_0(X)=\{B\in \mathfrak C(X): 0\in B\}$ and  $\mathfrak C_{00}(X)=\{B\in \mathfrak C(X): 0\in{\rm int }B\}$.
\begin{lemma}\label{5.2}
Assume that $X$ is a Banach space. Then

i) $\mathfrak C_0(X)$ is again a complete cone;

ii) $\mathfrak C_{00}(X)$ is a dense open cone of $\mathfrak C_0(X)$;

iii) $\mathfrak C_{00}(X)$ is an open cone of $\mathfrak C(X)$.
\end{lemma}
\begin{proof}
i) Clearly, $\mathfrak C_0(X)$ is again a  subcone of $\mathfrak C(X)$ and it is complete in the Hausdorff metric $d_H$.

ii) It is not difficult to observe that $\mathfrak C_{00}(X)$ is  dense in $\mathfrak C_0(X)$. Indeed, for every $\eps>0$, we have $C+\eps B_X\in\mathfrak C_{00}(X)$ with $d_H(C,C+\eps B_X)=\eps$. Thus, $\mathfrak C_{00}(X)$ is  dense in $\mathfrak C_0(X)$.  Since iii) implies  $\mathfrak C_{00}(X)$ is an open set of $\mathfrak C_0(X)$, it remains to show iii) that $\mathfrak C_{00}(X)$ is open in $\mathfrak C(X)$. Let $C\in\mathfrak C_{00}(X)$. Since $0\in{\rm int}C$, there is $\delta>0$ so that $\delta B_X\subset C$. Therefore, $\{D\in\mathfrak C(X): d_H(C,D)<\delta\}\subset\mathfrak C_{00}(X).$
Consequently, $\mathfrak C_{00}(X)$ is open in $\mathfrak C(X)$.
\end{proof}

Recall that  $\mathscr H^2(X)$ denotes the cone of all  positive quadratic homogenous continuous  convex functions  defined on $X$ endowed with the metric $d$ defined for $p,q\in\mathscr H(X)$ by
 $d(p,q)=\sup_{x\in B_X}|p(x)-q(x)|$, and that $\mathscr H^2_c(X)$ stands for the subcone of $\mathscr H^2(X)$ consisting of all    positive quadratic homogenous continuous coercive convex functions defined on $X$.

\begin{lemma}\label{5.3}
Let $X$ be a Banach space. Then

i)  $\mathscr H^2(X)$ is a complete metric cone;

ii) given $f\in\mathscr H^2(X)$ and $r>0$, $B_r (\equiv\{x\in X: f(x)\leq r\})$ is a convex body if and only if $f$ is coercive;

iii) $\mathscr H^2_c(X)$ is a dense open subset of $\mathscr H^2(X)$;

iv) given $f\in\mathscr H^2(X)$ and $r>0$,  $B_r$ is a strictly convex body if and only if $f$ is strictly convex and coercive;

v) for each $f\in\mathscr H^2(X)$, there is a closed convex set $C\subset X$ with $0\in{\rm int}C$ so that $f=\frac{1}2p_C^2$;

vi) the convex set $C$ defined in v) is a convex body if and only if $f$ is coercive.
\end{lemma}
\begin{proof}
i) Trivial.

ii) Let $f\in\mathscr H^2(X)$ be coercive and $r>0$. Since $f(x)\geq0$ with $f(0)=0$, and since $f$ is continuous and convex, $B_r\equiv\{x\in X: f(x)\leq r\}$ is a closed convex set with ${\rm int}B_r\neq\emptyset$. Coerciveness of $f$ entails that $B_r$ is bounded. Conversely, if $B_r$ is a convex body for some $r>0$, then coerciveness of $f$ follows from the homogeneity of $f$ and boundedness of $B_r$.

iii) For each $f\in\mathscr H^2(X)$ and $\eps>0$, let $g_\eps=\eps\|\cdot\|^2$. Then $f+g_\eps\in \mathscr H^2_c(X)$ with $d(f,g_\eps)\leq\eps$. Therefore, $\mathscr H^2_c(X)$ is dense in  $\mathscr H^2(X)$. Conversely, for each  $f\in\mathscr H^2_c(X)$, there is $\delta>0$ such that $f\geq\delta\|\cdot\|^2$. This implies that the open neighborhood $\{g\in\mathscr H^2(X): d(f,g)(\equiv\sup_{x\in B_X}|f(x)-g(x)|)<\delta\}$ of $f$ is contained in $\mathscr H^2_c(X)$.

iv) This is just Lemma \ref{3.6}.

v) Let $f\in\mathscr H^2(X)$. Since $f$ is a positive quadratic homogenous continuous  convex function, $p_C\equiv\sqrt{2f}$ is a continuous  positively homogenous function, i.e. a continuous coercive Minkowski functional. Clearly, $p_C$  is generated by $C\equiv\{x\in X:f(x)\leq\frac{1}2\}$. Therefore,  $f=\frac{1}2p_C^2$.

vi) It suffices to note that $C$ is bounded with $0\in{\rm int}C$ if and only if $p_C$ is continuous and coercive, which is equivalent to $f$ is continuous coercive.
\end{proof}
\begin{theorem}\label{5.4}
Suppose that $X$ is a Banach space. Then $T:\mathfrak C_{00}(X)\rightarrow\mathscr H^2_c(X)$ defined for $B\in\mathfrak C_{00}(X)$ by  $T(B)=\frac{1}2p^2_B$ is a fully order-reversing locally Lipschitz isomorphism.
\end{theorem}
\begin{proof}
We first show that $T: \mathfrak C_{00}(X)\rightarrow \mathscr H^2_c(X)$ defined for $B\in\mathfrak C_{00}(X)$ by  $T(B)=\frac{1}2p^2_B$ is fully order-reversing locally Lipschitz.

Since every $B\in\mathfrak C_{00}(X)$ is a convex body with $0\in{\rm int}(B)$, $T$  is  order-reversing. Conversely, for each $f\in\mathscr H^2_c(X)$, let $p=\sqrt{2f}$. Since $f$ is positive quadratic homogenous continuous  convex, $p$ is a continuous  coercive Minkowski functional. Consequently, $B\equiv\{x\in X: p(x)\leq1\}$ is a convex body with $0\in{\rm int}(B)$ so that $T(B)=f$. Therefore, $T: \mathfrak C_{00}(X)\rightarrow \mathscr H^2_c(X)$ is fully order-reversing.

 Next, we show that $T$ is locally Lipschitz.  For any fixed $A\in\mathfrak C_{00}(X)$, there is $\beta>0$ such that $2\beta B_X\subset A$.  Note that for each $\beta>\delta>0$, and for any $B\in\mathfrak C(X)$, $B\in\mathfrak A_{\delta}\equiv\{C\in\mathfrak C(X): d_H(A,C)<\delta\}$ entails that $\delta B_X\subset B\in\mathfrak C_{00}(X)$, and  that $B\subset A\oplus\delta B_X$ and $A\subset B\oplus\delta B_X$. We claim that $T$ is Lipschitz on the neighborhood $\mathfrak A_{\delta}$ of $A$.

For any $B, C\in \mathfrak A_{\delta}$, let $\alpha=d_H(B,C)$. Then $B\subset C\oplus\alpha B_X$ and $C\subset B\oplus\alpha B_X$. Consequently,  $p_{B\oplus\alpha B_X}\leq p_C$ and $p_{C\oplus\alpha B_X}\leq p_B$. Set $K_1=\{x\in B_X: p_B(x)\geq p_C(x)\}$ and $K_2=\{x\in B_X: p_B(x)\leq p_C(x)\}.$ Then
\[\sup_{x\in B_X}|p_B(x)-p_C(x)|=\max\Big\{\sup_{x\in K_1}(p_B(x)-p_C(x)),\sup_{x\in K_2}(p_C(x)-p_B(x))\Big\}\]
\[\leq\max \Big\{\sup_{x\in B_X}(p_B(x)-p_{B\oplus\alpha B_X}(x)), \sup_{x\in B_X}(p_C(x)-p_{C\oplus\alpha B_X}(x))\Big\}.\]
Note that for any $x(\neq0)\in X$, $\frac{p_B(x)}{\|x\|}\leq\frac{1}\delta,$ $\frac{x}{p_B(x)}\in{\text{bdr}}B$, and that $\frac{x}{p_B(x)}+\alpha\frac{x}{\|x\|}\in{\text{bdr}}(B\oplus\alpha B_X).$
Then \[p_{B\oplus\alpha B_X}(x)=\frac{p_B(x)\|x\|}{\|x||+\alpha p_B(x)}=\frac{p_B(x)}{1+\alpha\frac{p_B(x)}{\|x\|}}\geq\frac{p_B(x)}{1+\frac{\alpha}{\delta}}.\]
Therefore, $x\in K_1$ entails
\[0\leq p_B(x)-p_{B\oplus\alpha B_X}(x)\leq p_B(x)\Big(1-\frac{1}{1+\frac{\alpha}{\delta}}\Big)\leq p_B(x)\Big(\frac{\alpha}\delta\Big)=p_B(x)\Big(\frac{d_H(B,C)}\delta\Big).\]
We can show the following inequality in the same way:
For every $x (\neq0)\in K_2$,
\[0\leq p_C(x)-p_{C\oplus\alpha B_X}(x)\leq p_C(x)\Big(\frac{d_H(B,C)}\delta\Big).\]
Consequently,
\[d(T(B),T(C))=\sup_{x\in B_X}|T(B)(x)-T(C)(x)|\;\;\;\;\;\;\;\;\;\;\;\;\;\]\[=\frac{1}2\sup_{x\in B_X}|p^2_B(x)-p^2_C(x)|=
\frac{1}2\sup_{x\in B_X}\Big(p_B(x)+p_C(x)\Big)\cdot|p_B(x)-p_C(x)|\]
\[\leq M{d_H(B,C)},\]
where \[M=\frac{1}{2\delta}\max\Big\{\sup_{x\in B_X}p_B(x),\sup_{x\in B_X}p_C(x)\Big\}\Big(\sup_{x\in B_X}p_B(x)+\sup_{x\in B_X}p_C(x)\Big).\]

Finally, we show that $T^{-1}:\mathscr H^2_c(X)\rightarrow\mathfrak C_{00}(X)$ is again  fully order-reversing locally Lipschitz. By definition, $T^{-1}$ is fully order-reversing if and only if $T$ is fully order-reversing. It remains to show that $T^{-1}$ is locally Lipschitz on $\mathscr H^2_c(X)$.

Let $f\in \mathscr H^2_c(X)$. By Lemma \ref{5.3}, there is  $C\in\mathfrak C_{00}(X)$ such that $f=\frac{1}2p_C^2$. Let $\mathscr M_c(X)$ be the cone of all continuous coercive Minkowski functionals on $X$ endowed with the metric $d_{\mathscr M}$ which is defined for $p,q\in\mathscr M_c(X)$ by
$d_\mathscr M(p,q)=\sup_{x\in B_X}|p(x)-q(x)|$. Then $S: \mathscr H^2_c(X)\rightarrow\mathscr M_c(X)$ defined for $f\in\mathscr H^2_c(X)$ by $S(f)=\sqrt{2f}$ is locally Lipschitz. To show $T^{-1}$ is locally Lipschitz, it suffices to prove that the mapping $U: \mathscr M_c(X)\rightarrow \mathfrak C_{00}(X)$ defined for $p_C\in \mathscr M_c(X)$ by $U(p_C)=C$ is locally Lipschitz.

Note that the metric $d_\mathscr M$ on $\mathscr M_c(X)$ is defined for $g, h\in \mathscr M_c(X)$ by
$d_\mathscr M(g,h)=\sup_{x\in B_X}|g(x)-h(x)|$. Let $g=p_C$, $h=p_D$ with $d_\mathscr M(g,h)=r>0$. We can assume that  $C$ and $D$ are bounded by $\alpha>0$.

Positive homogeneity of $g$ and $h$ entails \[g\leq h+r\|\cdot\|, \;{\rm and\;}\;h\leq g+r\|\cdot\|,\] where $\|\cdot\|$ is the norm of $X$.
This in turn deduces that
\begin{equation}\label{5.5} C_r\equiv\{x\in X: h(x)+r\|x\|\leq1\}=\{x\in X: p_D(x)+r\|x\|\leq1\}\subset C\cap D\end{equation} and
\begin{equation}\label{5.8}D_r\equiv\{x\in X: g(x)+r\|x\|\leq1\}=\{x\in X: p_C(x)+r\|x\|\leq1\}\subset D\cap C.\end{equation}
On the other hand, $x\in X$, $h(x)=p_D(x)=1$ implies
\begin{equation}\label{5.7}1=p_{C_r}(\frac{x}{p_{C_r}(x)})= p_{C_r}(\frac{x}{p_{D}(x)+r\|x\|})=\frac{1}{1+r\|x\|}p_{C_r}(x)\geq\frac{1}{1+r\beta}p_{C_r}(x).\end{equation}
where $\beta=\sup_{x\in{\text{bdr}(D)}}\|x\|\; (>0)$. Therefore, \begin{equation}\label{5.8}D\subset(1+r\beta) C_r\subset(1+r\beta)(C\cap D)\subset(1+r\beta)C.\end{equation}  Analogously,
 \begin{equation}\label{5.9} C\subset(1+r\gamma)D_r\subset(1+r\gamma)(C\cap D)\subset(1+r\gamma)D,\end{equation}
 where $\gamma=\sup_{x\in{\text{bdr}(C)}}\|x\|.$
 Since $C, D$ are bounded by $\alpha$, it follows from (\ref{5.8}) and (\ref{5.9})
 \[d_H(C,D)\leq d_H(C,(1+r\beta)C)+d_H((1+r\beta)C, D)\]
 \[\leq d_H(C,(1+r\beta)C)+d_H\big((1+r\gamma)(1+r\beta)D,D\big)\]
 \[\leq r\beta\alpha+\big((1+r\gamma)(1+r\beta)-1)\big)\alpha\]
 \[=r(\beta+\gamma+r\beta\gamma)\alpha=L\cdot d_\mathscr M(g,h),\]
 where $L=(\beta+\gamma+r\beta\gamma)\alpha$.
 Thus, we have shown that the mapping $U: \mathscr M_c(X)\rightarrow \mathfrak C_{00}(X)$ is locally Lipschitz.
\end{proof}

\begin{lemma}\label{5.5}
Assume that $X$ is a Banach space admitting an equivalent strictly convex norm. Then the subset
$\mathfrak C_{sc,00}(X)$ consisting of all strictly convex bodies in $\mathfrak C_{00}(X)$ contains a dense $F_\sigma$-subset of
 $\mathfrak C_{0}(X)$.
\end{lemma}
\begin{proof}
Since $X$ has an equivalent strictly convex norm, without loss of generality, we can assume that $X=(X,\|\cdot\|)$ is itself strictly convex.   For each $n\in\N$, let $f_n=\frac{1}{2n}\|\cdot\|^2$. Then each of $\{f_n: n\in\N\}$ is continuous, strictly convex and coercive. By Lemma \ref{3.7}, for each $f\in\mathscr H^2(X)$, $g_n\equiv f+f_n, n=1,2,\cdots$ are strictly convex and coercive. Note that for each $n\in\N$, $\mathscr H^2_n(X)\equiv f_n+\mathscr H^2(X)$ is a closed cone of $\mathscr H^2(X)$.   Clearly, $\bigcup_{n=1}^\infty\mathscr H^2_n(X)$ is a dense $F_\sigma$-subset of $\mathscr H^2(X).$ Since $T:\mathfrak C_{00}(X)\rightarrow\mathscr H^2_c(X)$ is a locally Lipschitz isomorphism (Theorem \ref{5.4}), $T^{-1}(\mathscr H^2_n(X))$ is a closed subset of $\mathfrak C_{00}(X)$. Consequently, $\mathfrak F\equiv T^{-1}(\bigcup_{n=1}^\infty\mathscr H^2_n(X))=(\bigcup_{n=1}^\infty T^{-1}\mathscr H^2_n(X))$ (which consists of strictly convex bodies) is a dense $F_\sigma$-subset of $\mathfrak C_{00}(X)$. We finish the proof by noting that $\mathfrak C_{00}(X)$ is a dense open set of $\mathfrak C_{0}(X)$ (Proposition \ref{5.1} iii).
\end{proof}
Now, we restate and prove Theorem  \ref{1.1} as follows.
\begin{theorem}\label{5.10}
 In a Banach space $X$, the set $\mathfrak C_{sc}(X)$ consisting of all strictly convex bodies of $X$  is dense in $\mathfrak C(X)$ if and only if $X$ admits an equivalent strictly convex norm.
\end{theorem}
\begin{proof}
Necessity.  Assume that $\mathfrak C_{sc}(X)$   is dense in $\mathfrak C(X)$. Take any $B\in \mathfrak C_{sc}(X)$ with $0\in\text{int}B$, and let $C=B\bigcap(-B)$. Since $B$ is strictly convex, $C$ is again strictly convex.
Let $\||\cdot\||=p_C$, the Minkowski functional generated by $C$. Then it is easy to see that $\||\cdot\||$ is an equivalent strictly convex norm of $X$.

 Sufficiency. Suppose that $X$ admits an equivalent strictly convex norm. Then by Lemma \ref{5.5}, the set $\mathfrak C_{sc,00}(X)$  of all strictly convex bodies in $\mathfrak C_{00}(X)$ contains a dense $F_\sigma$ set of $\mathfrak C_{0}(X)$. We finish the proof by noting $\mathfrak C(X)=X+\mathfrak C_{00}(X)$, and $X+\mathfrak C_{sc,00}(X)=\mathfrak C_{sc}$.

\end{proof}
Recall that a Banach space $X$ is said to be  weakly compactly generated (WCG) if there is a weakly compact subset $W$ of $X$ such that ${\rm span}W$ is dense in $X$. Since every WCG space admits an equivalent strictly convex norm, and every separable Banach space is a WCG space (see, for instance, \cite{Di}), we have the following result.
\begin{corollary}\label{5.11}
If $X$ is a Banach space in one of the following classes, $$\mathfrak F=\{{\rm finite\;dimensional\;spaces}\},$$ $$\mathfrak S=\{{\rm separable\;Banach\;spaces}\},$$ $$\mathfrak S_{WCG}=\{{\rm weakly\;compactly\;generated\;spaces}\},$$ then the set consisting of all strictly convex bodies of $X$  is dense in $\mathfrak C(X)$.
\end{corollary}

\begin{corollary}\label{5.12}
If $X$ is a Banach space such that $X^*$ is $w^*$-separable, 
then  the set consisting of all strictly convex bodies of $X$  is dense in $\mathfrak C(X)$.
\end{corollary}
\begin{proof}
By Theorem \ref{5.10}, it suffices to note that every Banach space with $w^*$-separable dual admits an equivalent strictly convex norm.
\end{proof}

\begin{corollary}\label{5.13}
If $X$ is one of the following  Banach spaces: $c_0(\Gamma)$ for an arbitrary set $\Gamma$, $\ell_p\;(1\leq p\leq\infty)$, $L_p(\mu)\; (1\leq p\leq\infty,\mu\;{\rm is\;}\sigma\text{-\;finite})$ and $\ell_1[0,1]$,  then the set consisting of all strictly convex bodies of $X$  is dense in $\mathfrak C(X)$.
\end{corollary}
\begin{proof}
By Theorem \ref{5.10}, it is true for $X=c_0(\Gamma)$ because $c_0(\Gamma)$ is WCG for an arbitrary set $\Gamma$.
Again by Theorem \ref{5.10}, the conclusion holds for  $X=\ell_p\;(1\leq p<\infty)$, $L_p(\mu)\; (1\leq p<\infty,\mu\;{\rm is\;}\sigma\text{-\;finite})$ because they are separable.
According to Corollary \ref{5.11}, it is true when $X=\ell_\infty, L_\infty(\mu)$ because $\ell^*_\infty=\ell_1\oplus c_0^\bot$ is $w^*$-separable, and   $L_\infty(\mu)$ is isomorphic to $\ell_\infty$.
Since the dual $\ell^*_1[0,1]$ of $\ell_1[0,1]$ is $w^*$-separable, the conclusion for $X=\ell_1[0,1]$ follows from Corollary \ref{5.11}.
\end{proof}
\section{G\^{a}teaux smooth  convex bodies in $\mathfrak C(X)$}

In this section, we will show that if a Banach space $X$ admits an equivalent norm so that its dual $X^*$ is strictly convex with respect to the new dual norm, then G\^{a}teaux smooth convex bodies are dense in  $\mathfrak C(X)$, i.e. Theorem \ref{1.2} stated in Section 1.

 For a Banach space $X$, we again use  $\mathscr C_{\rm conv}(X)$ to denote the set of all extended real-valued lower semicontinuous proper convex functions on $X$,
 and $\mathscr C^*_{\rm conv}(X^*)$  the set of all extended real-valued $w^*$-lower semicontinuou proper convex functions on $X^*$. If $X$ is nontrivial, then both $\mathscr C_{\rm conv}(X)$  and $\mathscr C^*_{\rm conv}(X^*)$ are not cones. For example, given any $u\neq v\in X$, let $\delta_u,$ and $\delta_v$ be the indicators of the singletons $\{u\}$ and $\{v\}$, i.e. $\delta_u(x)=0,$ if $x=u;\;=+\infty,$ otherwise. Then $\delta_u,$ and $\delta_v$ are extended real-valued lower semicontinuous proper convex functions, but $\delta_u+\delta_v=+\infty.$ However, the following subsets of $\mathscr C_{\rm conv}(X)$ form cones.
 $\mathscr C_{\rm cconv}(X)$ denotes  the cone of all  lower   continuous coercive  convex functions on $X$, and  $\mathscr C^{*}_{\rm cconv}(X^*)$,  the cone of all  lower   continuous and $w^*$-lower semicontinuous coercive  convex functions on $X^*$.
Finally, $\mathscr H^{2}_{c}(X)$   ($\mathscr H^{*2}_{c}(X^*)$ , resp.) stands for the cone of all real-valued  continuous (continuous and $w^*$-lower semicontinuous, resp.) positive quadratic homogenous coercive  convex functions  on $X$ ($X^*$, resp.). 

 A convex body $B$ in $X^*$ is said to be {\it $w^*$-convex} if $B$ is $w^*$-closed.
 The following property is easy to follow.
\begin{proposition}\label{6.3}
 Let $C$ be a closed convex set of $X^*$ with $0\in{\rm int}(C)$. Then

 i) $p_C$ is $w^*$-lower semicontiuous if and only if $C$ is $w^*$-closed;

 ii) $C$ is a $w^*$-convex body if and only if $p_C$ is $w^*$-lower semicontinuous coercive.
 \end{proposition}
 Let $X$ be a Banach space. We use $\mathfrak C^*(X^*)$ to denote the cone of all $w^*$-compact convex sets of $X^*$ endowed with the Hausdorff metric $d_H$.
  $\mathfrak C^*_{0}(X^*)$ ($\mathfrak C^*_{00}(X^*)$, resp.) is the cone of all $w^*$-compact convex sets containing the origin ($w^*$-compact convex sets containing the origin in their interiors, resp.).

  Let $\mathscr H^{*2}(X^*)$ be the cone of all continuous and $w^*$-lower semicontinuous positive quadratic homogenous   convex functions endowed with the metric $d$ defined by
  \[d(f,g)=\sup_{x^*\in X^*,\|x^*\|\leq1}|f(x^*)-g(x^*)|,\;\;f,g\in \mathscr H^{*2}(X^*),\]
   and $\mathscr H^{*2}_c(X^*)$ ( $\mathscr H^{*2}_{\rm cs}(X^*)$, resp.) be the subcone consisting of all continuous positive quadratic  homogenous coercive  convex (strictly convex, resp.) functions  of $\mathscr H^{*2}(X^*)$.

   Parallel to Lemma \ref{5.2}, we have
\begin{lemma}\label{6.4}

i) Both $\mathfrak C(X^*)$ and $\mathfrak C^*_0(X^*)$ are complete cones;

ii) $\mathfrak C^*_{00}(X^*)$ is a dense open set of $\mathfrak C^*_0(X^*)$.
 \end{lemma}

Recall the definition of Fenchel's transform $\mathscr F:\mathscr C_{\rm conv}(X)\rightarrow\mathscr C_{\rm conv}(X^*)$ (Definition \ref{2.5.1}) is defined for $f\in \mathscr C_{\rm conv}(X)$ by
\[ \mathscr F(f)(x^*)=\sup_{x\in X}\{\langle x^*,x\rangle-f(x)\},\;\;x^*\in X^*. \]
The following property follows from definitions of subdifferential mapping of  convex functions and of Fenchel's transform.

\begin{lemma}\label{6.5}
Let $f$ be a lower semicontinuous proper convex function defined on a Banach space $X$, $x_0\in X$ and $x_0^*\in X^*$. Then
\begin{equation}\label{6.6}
x_0^*\in\partial f(x_0)\;\Longleftrightarrow\;f(x_0)+\mathscr F(f)(x_0^*)=\langle x^*_0,x_0\rangle.
\end{equation}
Consequently,
\begin{equation}\label{6.7}
x_0^*\in\partial f(x_0)\;\Longleftrightarrow\;x_0\in\partial\mathscr F(f)(x_0^*).
\end{equation}
\end{lemma}
\begin{proof}
By the definition  of the subdifferential mapping,
\[x_0^*\in\partial f(x_0)\Longleftrightarrow\;f(x)-f(x_0)\geq\langle x^*_0,x-x_0\rangle,\;\forall\;x\in X\]
\[\Longleftrightarrow \langle x^*_0,x_0\rangle- f(x_0)\geq \langle x^*_0,x\rangle-f(x),\;\forall\;x\in X\]
\[\Longleftrightarrow \mathscr F(f)(x_0^*)=\langle x^*_0,x_0\rangle- f(x_0) \]
\[\Longleftrightarrow \mathscr F(f)(x_0^*)+ f(x_0)=\langle x^*_0,x_0\rangle. \]
\end{proof}
Recall that for a subset $A\subset X$, its polar $A^\circ$ is defined by \[A^\circ=\{x^*\in X^*: \langle x^*,x\rangle\leq1,\;\forall x\in A\};\] and the support function $\sigma_A$ of $A$ is defined for $x^*\in X^*$ by \[\sigma_A(x^*)=\sup\{\langle x^*.x\rangle: x\in A\}.\]
The following property is easy to obtain.
\begin{proposition}\label{6.8}
Let $B\subset\mathfrak C_{00}(X)$.  Then its polar $B^\circ\in\mathfrak C_{00}(X^*)$, and  the Minkowski functional $p_{B^\circ}$ satisfies $p_{B^\circ}=\sigma_B.$
\end{proposition}
\begin{lemma}\label{6.9}
Suppose that $X$ is a Banach space. Then

i) for each $f\in \mathscr H^{2}_{c}(X)$ there is a convex body $B\in\mathfrak C_{00}(X)$ such that
$f=\frac{1}2p_B^2$, where $p_B$ is the Minkowski functional generated by $B$;

ii) for each $f\in \mathscr H^{2}_{c}(X)$, we have $\mathscr F(f)=\frac{1}2p^2_{B^\circ}=\frac{1}2\sigma^2_{B}$;

iii) $\mathscr F: \mathscr H^{2}_{c}(X)\rightarrow \mathscr H^{2*}_{c}(X)$ is a fully order-reversing isomorphism.
\end{lemma}
\begin{proof}
i) is just Lemma \ref{5.3} v).

ii) By i), for each $f\in \mathscr H^{2}_{c}(X)$ there is a convex body $B\in\mathfrak C_{00}(X)$ such that
$f=\frac{1}2p_B^2$. Fix any $x_0\in X$ and $x^*_0\in\partial f(x_0)=p_B(x_0)\partial p_B(x_0)$. Then $x^*_0=p_B(x_0)z^*_0$ for some $z^*_0\in\partial p_B(x_0)$.
Definition of  Fenchel's transform entails
\[\mathscr F(f)(x^*_0)=\sup\{\langle x^*_0,x\rangle-f(x): x\in X\}=\sup\{p_B(x_0)\langle z^*_0,x\rangle-\frac{1}2p_B^2(x): x\in X\}.\]
Since $g(x)\equiv \frac{1}2p_B^2(x)-p_B(x_0)\langle z^*_0,x\rangle$ is a continuous convex function with $0\in\partial g(x_0)$,
\begin{equation}\label{6.10}
{\rm min}_{x\in X}g(x)=g(x_0)=\frac{1}2p_B^2(x_0)-p_B(x_0)\langle z^*_0,x_0\rangle.
\end{equation}
Note that $z^*_0\in\partial p_B(x_0)$ implies $\langle z^*_0,x_0\rangle=p_B(x_0)$. This and (\ref{6.10}) entail
\begin{equation}\label{6.11}\mathscr F(f)(x^*_0)=\sup\{-g(x):x\in X\}=-\min_{x\in X}g(x)=-g(x_0)=\frac{1}2p_B^2(x_0).
\end{equation}

On the other hand, note that there is $z_0^*\in\partial p_B(x_0)$ such that $x_0^*=p_B(x_0)z_0^*$. It follows from (\ref{6.11})
\[\frac{1}2p_B^2(x_0)=\mathscr F(f)(x^*_0)=\sup\{\langle x^*_0,x\rangle-f(x): x\in X\}\]
\[=\sup\{p_B(x_0)\langle  z^*_0,x\rangle-\frac{1}2p^2_B(x): x\in X\}.\]
This and $\langle  z^*_0,x\rangle=p_B(x_0)$ imply
 \[\frac{1}2p_B^2(x_0)=\sup\{p_B(x_0)\langle  z^*_0,x\rangle-\frac{1}2p^2_B(x): x\in X, p_B(x)\leq p_B(x_0)\}\]
 \[\frac{1}2p_B^2(x_0)=\sup\{p_B(x_0)\langle  z^*_0,x\rangle-\frac{1}2p^2_B(x): x\in X, p_B(x)= p_B(x_0)\}\]
 \[=\sup\{p_B(x_0)\langle  x^*_0,x\rangle: x\in X, p_B(x)\leq 1\}-\frac{1}2p^2_B(x_0)\]
 \[=p_B(x_0)\sigma_B(x_0^*)-\frac{1}2p^2_B(x_0),\]
 where $\sigma_B(x_0^*)=\sup\{\langle  x^*_0,x\rangle: x\in B\}.$ This entails that \begin{equation}\label{6.12}p_B(x_0)=\sigma_B(x_0^*),\;\;\forall\;x_0\in X,\;x_0^*\in\partial f(x_0).\end{equation}
This combined with (\ref{6.11}) derive
\begin{equation}\label{6.13}\mathscr F(f)(x^*_0)=\frac{1}2\sigma^2_B(x_0^*),\;\forall x_0^*\in\partial f(X).\end{equation}
Since $f$ is continuous coercive convex, by an argument of the Bishop-Phelps theorem (see, for instance, \cite{Ph}), we see the range $\partial f(X)$ of the  subdifferential mapping $\partial f$ is dense in $X^*$.
This, local Lipschitz continuity of $\mathscr F(f)$ on $X^*$ and (\ref{6.13}) together imply
\[\mathscr F(f)(x^*)=\frac{1}2\sigma^2_B(x^*),\;\forall x^*\in X^*.\]

 To finish the proof of of ii), it suffices to note $p_{B^*}=\sigma_B$ on $X^*$.

 iii). By ii), it is easy to check that $\mathscr F: \mathscr H^{2}_{c}(X)\rightarrow \mathscr H^{2*}_{c}(X)$ is a fully order-reversing isomorphism.
\end{proof}
\begin{remark}
Indeed, we can show that $\mathscr F: \mathscr H^{2}_{c}(X)\rightarrow \mathscr H^{2*}_{c}(X)$ is also a locally Lipschitz isomorphism.
\end{remark}

\begin{lemma}\label{6.14}
 Let $X$ be a Banach space.


 i) If $g\in\mathscr H^{*2}_{c}(X^*)$ is strictly convex, then $f\equiv\mathscr F^{-1}(g)\in \mathscr H^{2}_{c}(X)$ is everywhere G\^{a}teaux differentiable.

 ii) If $g\in\mathscr H^{*2}_{c}(X^*)$ is everywhere G\^{a}teaux differentiable, then $f\equiv\mathscr F^{-1}(g)\in \mathscr H^{2}_{c}(X)$ is strictly convex.

\end{lemma}
\begin{proof}
 Let $g\in\mathscr H^{*2}_{c}(X^*)$. By Lemma \ref{6.9} ii), there is a convex body $B\in\mathfrak C_{00}(X)$ with $B^\circ\in\mathfrak C^*_{00}(X^*)$ such that
$g=\frac{1}2\sigma_B^2=\frac{1}2p_{B^\circ}^2$ and  $f=\mathscr F^{-1}(g)=\frac{1}2p_{B}^2$.

i) It follows from  (\ref{6.12}) and (\ref{6.13}) that for each $x_0\in X$, we have
\[p_B(x_0)=\sigma_B(x_0^*),\;x_0^*\in\partial f(x_0),\]
and \[g(x^*_0)=\frac{1}2\sigma^2_B(x_0^*)=\frac{1}2p^2_B(x_0),\; x_0^*\in\partial f(x_0).\]
Therefore, $g$ is  constant on $\partial f(x_0)$. Convexity of the set $\partial f(x_0)$ and strict convexity of $g$ entail that $\partial f(x_0)$ is a singleton, which is equivalent to that $f$  is  G\^{a}teaux differentiable at $x_0$.

ii) Suppose, to the contrary, that $f$ is not strictly convex. Then by Lemma \ref{3.4}, there is $r>0$ such that $B_r\equiv\{x\in X: f(x)\leq r\}$ is not strictly convex. Since $f$ is positively quadric  homogenous, we can assume that $r=\frac{1}2$, that is, $B_r=B$. Therefore, there exist $u\neq v\in{\rm bdr}(B)$ such that the segment $[u,v]\equiv\{\lambda u+(1-\lambda)v: 0\leq\lambda\leq1\}\subset B.$ By Zorn's lemma, there is a maximal convex set $K$ of ${\rm bdr}(B)$ such that $[u,v]\subset K$. It is trivial to get that $K$ is closed with $K\cap{\rm int}B=\emptyset.$ By the Hahn-Banach separation theorem of convex sets, there exists $x^*_0\in X^*$ such that
\[\langle x_0^*,x\rangle=1,\;\forall\;x\in K\;\;{\rm and\;}\;\langle x_0^*,x\rangle<1,\;\forall\;x\in B\setminus K.\]
This implies that $K\subset\partial\sigma_B(x_0^*)$. It follows from $\sigma_B(x_0^*)=1$ that \[K\subset\sigma_B(x^*_0)\partial\sigma_B(x_0^*)=\partial(\frac{1}2\sigma^2_B(x_0^*))=\partial g(x_0^*).\] This is a contradiction to that $g$ is G\^{a}teaux differentiable at $x^*_0$.
\end{proof}
\begin{lemma}\label{6.15}
Let $X$ be a Banach space admitting an equivalent norm  so that $X^*$ is strictly convex with respect to the new norm. Then
the cone $\mathfrak C_{\rm 00,gsm}(X)$ consisting of all G\^{a}teaux smooth convex bodies $B$ with $0\in{\rm int}(B)$ contains a dense open subset of
$\mathfrak C_{0}(X)$.
\end{lemma}
\begin{proof}
Without loss of generality, we can assume that $X^*$ is itself strictly convex. Let $f_n=\frac{1}n(\|\cdot\|^*)^2,\;n=1,2,\cdots$, where $\|\cdot\|^*$ is the dual norm of $X^*$.
Then $f_n\in\mathscr H^{*2}_{cs}(X^*)$ for all $n\in\N$. Therefore, $\mathscr H^{*2}_n(X^*)\equiv f_n+\mathscr H^{*2}(X^*)\subset\mathscr H^{*2}_{cs}(X^*)$. Clearly, $\mathfrak D\equiv\bigcup_{n\in\N} \Big(f_n+\mathscr H^2(X^*)\Big)$ is a dense $F_\sigma$-set of $\mathscr H^2(X^*)$. By Lemma \ref{6.9} iii), $\mathscr F^{-1}: \mathscr H^{*2}_{c}(X^*) \rightarrow \mathscr H^{2}_{c}(X)$ is a continuous fully order-reversing isomorphism. Consequently,  $\mathfrak S\equiv\mathscr F^{-1}(\mathfrak D)$ is a dense  subset of $\mathscr H^{2}_{c}(X)$. Since each $g\in\mathfrak D$ is strictly convex, by Lemma \ref{6.14} ii), $f\equiv\mathscr F^{-1}(g)$ is everywhere G\^{a}teaux differentiable. Let $p$ be the Minkowski functional so that $f=\frac{1}2p^2$. Then the convex body
\begin{equation}\label{6.16}B=\{x\in X: p(x)\leq1\}=\{x\in X: f(x)\leq\frac{1}2\}\end{equation} is G\^{a}teaux smooth. Since the correspondence (from $f\in\mathscr H^{2}_{c}(X)$ to $B\in\mathfrak C_{00}(X)$) defined by (\ref{6.16})) is a continuous isomorphism from $\mathscr H^{2}_{c}(X)$ to $\mathfrak C_{00}(X)$, G\^{a}teaux smooth convex bodies are dense in $\mathfrak C_{00}(X)$. Consequently,  $\mathfrak C_{\rm 00,gsm}(X)$ (the cone of G\^{a}teaux smooth convex bodies) is dense in $\mathfrak C_{0}(X)$.

\end{proof}

Now, we are ready to prove Theorem \ref{1.2}. First, we restate it as follows.
\begin{theorem}\label{6.17}
Let $X$ be a Banach space admitting an equivalent norm  so that $X^*$ is strictly convex with respect to the new norm. Then
the cone $\mathfrak C_{\rm sm}(X)$ consisting of all G\^{a}teaux smooth convex bodies  contains a dense  subset of
$\mathfrak C(X)$.
\end{theorem}

\begin{proof} By Lemma \ref{6.16}, the cone $\mathfrak C_{\rm 00,sm}(X)$ consisting of all G\^{a}teaux smooth convex bodies $B$ with $0\in{\rm int}(B)$ contains a dense  subset $\mathfrak S$ of
$\mathfrak C_{0}(X)$. This entails that $X+\mathfrak S$ is dense in $X+\mathfrak C_{0}(X)=\mathfrak C(X).$

\end{proof}
\section{Strictly convex and G\^{a}teaux smooth convex bodies }
In this section, we will show that if a Banach space $X$ admits an equivalent strictly convex norm so that its dual norm is also strictly convex on $X^*$, then every convex body of $X$ can be uniformly approximated by strictly convex and G\^{a}teaux smooth convex bodies, i.e. Theorem \ref{1.3} stated in Section 1. In particular, every convex body in a separable Banach space, or, a reflexive Banach space  is uniformly approximated by strictly convex and G\^{a}teaux smooth convex ones, i.e. Theorem \ref{1.4}.

 Again as before, we restate Theorem \ref{1.3} as follows.
\begin{theorem}\label{7.1}
Let $X$ be a Banach space admitting an equivalent strictly convex norm so that its dual is also strictly convex. Then every convex body of $X$ can be uniformly approximated by strictly convex and G\^{a}teaux smooth convex bodies.
\end{theorem}
\begin{proof} Our proof was motivated by Asplund's averaging technique \cite{Asp}.
Clearly, it suffices to show that it is true for every $B\in\mathfrak C_{00}(X)$. We can assume that the original  norm $\|\cdot\|$ of $X$ itself is strictly convex and its dual  norm $\|\cdot\|^*$ on $X^*$ is also strictly convex.

Fix any  $B\in\mathfrak C_{00}(X)$. We can claim that for every $\eps>0$, there exist a strictly convex body $A\in\mathfrak C_{00}(X)$ contained in $B$ and a  convex body $C\in\mathfrak C_{00}(X)$ containing $B$ so that its polar $C^\circ=\{x^*\in X^*: \langle x^*,x\rangle\leq1;\;\forall x\in C\}$ is a $w^*$-closed strictly convex body of $X^*$ such that  $d_{H}(A,C)<\eps.$ Indeed, let $p_A, p_B$ and $p_C$ be successively, the Minkowski functionals generated by $A, B$ and $C$. Then for any $\delta>0$, both $f_\delta\equiv\frac{1}2(\delta\|\cdot\|^2+p_A^2)$ and $g_\delta^*\equiv\frac{1}2(\delta\|\cdot\|{^*}^2+p_{C^*}^2)$ are strictly convex, and $g_\delta^*$ is $w^*$-lower semicontinuous on $X^*$.
Put $A_\delta=\{x\in X: f_\delta(x)\leq1\}$ and $C_\delta=(C_\delta^*)^\circ$. Clearly, $A_\delta\subset B\subset C_\delta$ and
$d_{H}(A_\delta,C_\delta)<\eps$ for all sufficiently small $\delta>0$.

Starting with $f_0=\frac{1}2p_A^2, g_0=\frac{1}2p_C^2$,\;let \[f_n=\frac{1}2(f_{n-1}+g_{n-1}), \;g_n=f_{n-1}\Box g_{n-1},\; n=1,2,\cdots,\] where $f\Box g$ is the inf-convolution of $f$ and $g$ defined in Section 2.
Then for all $n\in\N$, \[f_0\geq f_1\geq\cdots\geq f_{n-1}\geq f_n\geq g_n\geq g_{n-1}\geq\cdots\geq g_0,\]
and
\[\mathscr F(f_0)\leq \mathscr F(f_1)\leq\cdots\leq \mathscr F(f_{n-1})\leq\mathscr F(f_n)\leq \mathscr F(g_n)\leq \mathscr F(g_{n-1})\leq\cdots\leq \mathscr F(g_0).\]
 Note that $f_0$ is strictly convex on $X$  and $\mathscr F(g_0)=\frac{1}2p_{C^*}^2$ is strictly convex on $X^*$, and that $\mathscr F(g_n)=\frac{1}2(\mathscr F(f_{n-1})+\mathscr F(g_{n-1}))$. Then for each $n\in\mathbb N$, $f_n$ is strictly convex on $X$ and $\mathscr F(g_n)$ is strictly convex on $X^*$. Monotonicity of both the sequences $\{f_n\}$ and $\{g_n\}$ and $\sup_{x\in B_X}|f_n(x)-g_n(x)|\rightarrow 0$ entail that that there is a convex body $D\in\mathfrak C_{00}$ such that
 \[\lim_nf_n=\frac{1}2p_D^2=\lim_ng_n,\;\;\text{uniformly on}\;B_X.\]
 Clearly, $A\subset D\subset C$. Since $A\subset B\subset C$ with $d_{H}(A,C)<\eps$, we get $d(B,D)<\eps$. By an argument as the same as the proof of Theorem 1 in \cite{Asp}, we see that $\frac{1}2p_D^2$ is strictly convex on $X$, and $\mathscr F(\frac{1}2p_D^2)$ is strictly convex on $X^*$. It follows that $D$ is both strictly convex and G\^{a}teaux smooth.
\end{proof}
The following consequence is just Theorem \ref{1.4}.
\begin{corollary}
Suppose that $X$ is either separable, or, reflexive. Then every convex body can be uniformly approximated by strictly convex and G\^ateaux smooth convex bodies.
\end{corollary}

\begin{proof}
By Theorem \ref{7.1}, it suffices to note that every separable (reflexive, respectively) Banach space admits an equivalent strictly convex norm so that it dual norm is also strictly convex on $X^*$.
\end{proof}

\begin{remark}
We do not know whether G\^ateaux smooth convex bodies are dense in $\mathfrak C(X)$ whenever $X$ is a  G\^ateaux smooth Banach space; and whether it can be deduced that strictly convex and G\^ateaux smooth convex bodies are dense in $\mathfrak C(X)$ if $X$ is both strictly convex and G\^ateaux smooth.
\end{remark}

\bibliographystyle{amsalpha}

\begin{thebibliography}{AcBeRu}

\bibitem{AAM} S. Artstein-Avidan, V.D.  Milman, {\it The concept of duality in convex analysis, and the characterization of the Legendre transform.}  Ann. of Math., 2009, 169 (2): 661-674.
\bibitem{Asp} E. Asplund, {\it Averaged norms.} Israel J. Math. 5 (1967), 227-233.
\bibitem{Asp1} E. Asplund, {\it Fr\'echet differentiability of convex functions.} Acta Math. 121(1967), 31-47.
\bibitem{bl2017} I. B\'{a}r\'{a}ny, M.  Laczkovich, {\it Magic mirrors, dense diameters, Baire category}, J. Convex Anal. {\bf 24} (2017), no. 1, 93-102.

\bibitem{barany2015}  I. B\'{a}r\'{a}ny,  R. Schneider, {\it Typical curvature behaviour of bodies of constant width}, Adv. Math. {\bf 272} (2015), 308-329.

\bibitem{bz1990} I. B\'{a}r\'{a}ny, T. Zamfirescu, {\it Diameters in typical convex bodies}, Canad. J. Math. {\bf 42} (1990), no. 1, 50-61.

\bibitem{CV} C. Castaing , M. Valadier, {\it Convex Analysis and Measurable Multifunction.} Lect. Notes in Math.  vol. 580, Springer-Verlag, 1977.
\bibitem{CL} L. Cheng, S. Luo, {\it On order-preserving and order-reversing mappings
defined on cones of convex functions. } Sci. China. Math. Vol. 64 (2021) No. 8: 1817-1842.
https://doi.org/10.1007/s11425-018-1707-x.
\bibitem{CL2} L. Cheng, S. Luo, {\it A localization of the Artstein-Avidan-Milman theorem.} J. Math. Anal. Appl.   513 (2022)  no. 2, Paper No. 126239, 12 pp.

\bibitem{DGZ} R. Deville, G. Godefroy, V. Zizler,
{\it Smoothness and renormings in Banach spaces}, Pitman Monogr. Surveys Pure Appl. Math., 64
Longman Scientific \& Technical, Harlow; copublished in the United States with John Wiley \& Sons, Inc., New York.

\bibitem{Di} J. Distel, {\it Geometry of Banach spaces}, Lect. Notes in Math. vol. 485, Springer-Verlag, 1975.




\bibitem{EL} {I. Ekeland, G. Lebourg,}
{\it Generic Fr\'echet-differentiability and perturbed optimization problems in Banach spaces.} Trans. Amer. Math. Soc. 224 (1976), no.2, 193-216.


\bibitem {En} P. Enflo, {\it Banach spaces which can be given an equivalent uniformly convex norm}, Israel J. Math., 1972, 13: 281--288.



\bibitem{Fa} M. Fabian,  {\it G\^{a}teaux differentiability of convex functions and topology.}
Canad. Math. Soc. Ser. Monogr. Adv. Texts Wiley-Intersci. Publ.
John Wiley \& Sons, Inc., New York, 1997.

\bibitem{gruber1977} P. Gruber, {\it Die meisten konvexen K\"{o}rper sind glatt, aber nicht zu glatt}, Math. Ann. {\bf 229} (1977), no. 3, 259-266.

\bibitem{IRS} {A.N. Iusem, D. Reem, B.F. Svaiter}, \emph{Order preserving and order reversing operators on the class of convex functions in Banach spaces.} J. Funct. Anal., 2015, 268 (1):  73-92.

\bibitem{Ja} R.C. James,  \emph{Super-reflexive Banach spaces},
Canad. J. Math.   24 (1972), 896--904.

\bibitem{Klee1959} V. Klee,  {\it Some new results on smoothness and rotundity in normed linear spaces}, Math. Ann. {\bf 139} (1959), 51-63.

\bibitem{Ph} {R.R. Phelps}, \emph{ Convex Functions, Monotone Operators and Differentiability.} Lect. Notes in Math. vol. 1364, Springer-Verlag, 1989.
\bibitem{Ph1} {R.R. Phelps}, \emph{ Some topological properties of support points of convex sets.} Israel J. Math. 13 (1972), 327-336.

\bibitem{Pi} G. Pisier,
  \emph{Martingales with values in uniformly convex spaces}
 Israel J. Math.  20(1975), 326--350.
\bibitem{PPN} {D. Preiss, R.R. Phelps, I. Namioka},
\emph{Smooth Banach spaces, weak Asplund spaces and monotone or usco mappings.}
Israel J. Math. 72 (1990), no. 3, 257-279.

\bibitem{Schneider2015} R. Schneider, {\it Curvatures of typical convex bodies. The complete picture},  Proc. Amer. Math. Soc. {\bf 143} (2015), no. 1, 387-393.

\bibitem{zm1980-2} T. Zamfirescu, {\it Nonexistence of curvature in most points of most convex surfaces}, Math. Ann. {\bf 252} (1980), no. 3, 217-219.

\bibitem{zm1982} T. Zamfirescu, {\it Many endpoints and few interior points of geodesics}, Invent. Math. {\bf 69} (1982), no. 2, 253-257.

\bibitem{zm1987} T. Zamfirescu, {\it Nearly all convex bodies are smooth and strictly convex}, Monatsh. Math. {\bf 103} (1987) 57-62.

 \bibitem{zm1992}T. Zamfirescu, {\it Long geodesics on convex surfaces}, Math. Ann. {\bf 293} (1992), no. 1, 109-114.

\end{thebibliography}

\end{document}